\documentclass[11pt, letterpaper,final]{amsart}
\usepackage{amsfonts,amsmath, amsthm, amssymb, latexsym}
\usepackage[english]{babel}
\usepackage[utf8]{inputenc}
\usepackage[all]{xy}
\usepackage{xspace}
\usepackage{amscd}
\usepackage{comment}
\usepackage{setspace}
\usepackage{enumerate}
\usepackage{stmaryrd}
\usepackage{xcolor}
\usepackage{tikz-cd}
\usepackage{hyperref}
\usepackage[mathscr]{eucal}

\usepackage[hmargin=3cm,vmargin=3cm]{geometry}

\newcommand{\Aut}{\operatorname{Aut}}

\newcommand{\Zp}{\mathbb Z_{\geq 0}}
\newcommand{\Lk}[1]{L_k(#1)}

\newcommand{\spn}{\operatorname{span}_k}
\newcommand{\Mmult}[2]{\mathsf M(#1,#2)}
\newcommand{\SSE}{\sim_{\mathrm{SSE}}}
\newcommand{\SE}{\sim_{\mathrm{SE}}}

\newcommand{\ESSE}{\sim_{\mathrm{ESSE}}}
\newcommand{\Kgr}{K_0^{\mathrm{gr}}}
\newcommand{\one}{\mathbf 1}
\newcommand{\transpose}{\mathsf t}

\newtheorem{thm}{Theorem}[section]
	\newtheorem{cor}[thm]{Corollary}
	\newtheorem{lem}[thm]{Lemma}
	\newtheorem{prp}[thm]{Proposition}

	\newtheorem{theoremintro}{Theorem}
	
    \newtheorem{corollaryintro}[theoremintro]{Corollary}

\theoremstyle{definition}

\begin{document}
	\title[Graded classification of Leavitt path algebras]{Graded classification of Leavitt path algebras \\ in terms of strong shift equivalence}
	
	        \author{Boris Bilich}
        \address{Department of Mathematics, University of Haifa, Haifa, Israel, and Department of Mathematics, University of G\"ottingen, G\"ottingen, Germany.}
        \email{bilichboris1999@gmail.com}
        
        \author{Adam Dor-On}
        \address{Department of Mathematics, University of Haifa, Mount Carmel, Haifa 3103301, Israel.}
        \email{adoron.math@gmail.com}

        \date{\today}
	\subjclass[2020]{}
	


\thanks{B. Bilich was partially supported by a Bloom PhD scholarship at Haifa University. A. Dor-On was partially supported by an NSF-BSF grant no. 2350543 / 2023695 (respectively), by an NSF / BSF grant no. 2452324 / 2024734 (respectively), and a DFG Middle-Eastern collaboration project no. 529300231.}

\begin{abstract}
Given two finite essential adjacency matrices $A$ and $B$, Hazrat's graded classification conjectures posit that an order preserving $\mathbb{Z}[x,x^{-1}]$-module isomorphism of $K_0$ groups implies graded Morita equivalence of the Leavitt path algebras of $A$ and $B$, while the pointed version predicts a graded isomorphism of the Leavitt path algebras when the $K_0$ group isomorphism additionally preserves the class of the regular module. For any field $k$, we show that the Leavitt path algebras over $k$ of $A$ and $B$ are graded Morita equivalent if and only if $A$ and $B$ are strong shift equivalent. By appealing to counterexamples of Kim and Roush from symbolic dynamics, this shows that Hazrat's graded classification conjectures are false.
\end{abstract}

\maketitle

\section{Introduction}

Leavitt path algebras occupy an especially fruitful point of contact between pure algebra, operator algebras, and symbolic dynamics. We will say that a square matrix $A$ with entries in $\mathbb Z_{\geq 0}$ indexed by a finite set $V_A$ is \emph{essential} if it has no zero rows and no zero columns. We write $G_A$ for the directed graph with vertex set $V_A$ and exactly $A_{uv}$ edges from $u$ to $v$. Thus, $A$ gives rise on the one hand to a shift of finite type and, on the other hand, to the Leavitt path algebra $L_k(G_A)$. Leavitt path algebras were introduced in their graph-theoretic form by Abrams and Aranda Pino \cite{AbramsArandaPino} and by Ara, Moreno, and Pardo \cite{AraEtAlNonstable} (see also the monograph \cite{AAMS}), and are purely algebraic counterparts of the Cuntz-Krieger graph $C^*$-algebra which retains the Cuntz--Krieger relations and canonical grading \cite{CuntzKrieger, CuntzReducible}. This common ground between C*-algebras and pure algebra has led to a rather striking phenomenon. Structural and classification results on one side frequently have analogues on the other, while the points at which the analogy breaks tend to reveal subtle information about the underlying dynamics (see the recent survey of Corti\~nas and Hazrat \cite{CortinasHazratSurvey} for a broader account).

The problem considered here arises from one of the most basic
classification questions in symbolic dynamics. In his foundational work \cite{Williams}, Williams translated the conjugacy problem for shifts of finite type into a problem concerning factorizations of nonnegative integer matrices. Two square nonnegative integer matrices $A$ and $B$ are said to be \emph{elementary strong shift related} if there are rectangular matrices $R,S$ over $\mathbb Z_{\geq 0}$ such that
$$
 A=RS,\ \ \text{and} \ \ B=SR.
$$
Strong shift equivalence is then the transitive closure of elementary strong shift relation. Williams proved that the shifts of finite type defined by essential matrices are conjugate precisely when their associated matrices are strong shift equivalent.

Williams also introduced the weaker relation of \emph{shift
equivalence}. $A$ and $B$ are said to be shift equivalent if, for some lag $\ell\geq 1$, there are rectangular nonnegative integer matrices $R,S$ such that
\begin{equation*}
 AR=RB,\ \ \ \ BS=SA,\ \ \ \ RS=A^\ell,\ \ \ \ SR=B^\ell.
\end{equation*}
Strong shift equivalence always implies shift equivalence, and the latter relation is governed by Krieger's dimension triple \cite{Krieger,LindMarcus}, and is decidable by work of Kim and Roush \cite{KimRoushDecidabilityI,KimRoushDecidabilityII}. Williams conjectured that shift equivalence and strong shift equivalence coincide. However, Kim and Roush disproved Williams' conjecture for reducible matrices \cite{KimRoushReducible}, and then for primitive matrices \cite{KimRoush}. At the same time, the decidability problem for strong shift equivalence, and hence for conjugacy of shifts of finite type, remains one of the central open problems in symbolic dynamics (see \cite{BoyleSchmieding,JeandelSearch} for recent perspectives).

In tandem with the early development of this theory, Cuntz and Krieger associated to a finite matrix $A$ a $C^*$-algebra $\mathcal O_A$ carrying a canonical gauge action and a canonical diagonal subalgebra \cite{CuntzKrieger,CuntzReducible}. Their construction initiated a fruitful interaction between symbolic dynamics and operator algebras. Krieger's dimension triple is recovered from equivariant $K$-theory, so a gauge-preserving stable isomorphism can only occur when the matrices are shift equivalent. For primitive matrices, Bratteli and Kishimoto proved the converse using deep results on trace-scaling automorphisms of stable AF-algebras \cite{BratteliKishimoto}. More recently, shift equivalence has been characterized by gauge-equivariant stable homotopy equivalence of the corresponding graph $C^*$-algebras \cite{BilichDorOnRuizHomotopy}. Thus, the hierarchy between shift equivalence and strong shift equivalence has several natural graph-algebraic shadows. For instance, preserving the diagonal subalgebra via the graded equivalence recovers conjugacy \cite{CarlsenRoutCstar, CarlsenRoutSteinberg}, while weaker graded equivalences seem to retain only the eventual, dimension-group data.

Hazrat proposed that the natural grading on a Leavitt path algebra might provide an algebraic counterpart to this picture. The algebra $L_k(G_A)$ has its standard $\mathbb Z$-grading, with vertices in degree zero, edges in degree one, and ghost edges in degree minus one. Its graded Grothendieck group $K_0^{\mathrm{gr}}(L_k(G_A))$ is a partially ordered $\mathbb Z[x,x^{-1}]$-module, where the action of $x$ is induced by suspension. For a finite sinkless sourceless graph this module
is naturally identified, up to the transpose dictated by the adjacency convention, with Krieger's dimension triple. In particular, by \cite{HazratDynamics,HazratClassification} (see also
\cite{HazratBook,CortinasHazratSurvey}) we have that, as ordered $\mathbb Z[x,x^{-1}]$-modules, 
\begin{equation}\label{eq:Kgr-SE-intro}
 K_0^{\mathrm{gr}}(L_k(G_A))\cong
 K_0^{\mathrm{gr}}(L_k(G_B))
 \quad\Longleftrightarrow\quad
 A\sim_{\mathrm{SE}}B
\end{equation}
and the class $[L_k(G_A)]$ of the regular module corresponds to a distinguished order unit. The positive cone of the graded Grothendieck group also has a graph-theoretic presentation as the talented monoid. This monoid-theoretic viewpoint has been used to recover substantial graph-theoretic, ideal-theoretic, and growth information from
$K_0^{\mathrm{gr}}$ \cite{HazratLiTalented,CordeiroGoncalvesHazratTalented,
HazratSebandalVilela,CordeiroGoncalvesHazratUniform}.

This led to two closely related forms of the Graded Classification
Conjecture. The pointed form predicts that an order-preserving
$\mathbb Z[x,x^{-1}]$-module isomorphism which carries
$[L_k(G_A)]$ to $[L_k(G_B)]$ is induced by a graded isomorphism
$L_k(G_A)\cong L_k(G_B)$. The unpointed, Morita-theoretic form predicts that for finite essential matrices $A$ and $B$ we have
\begin{equation*}
 A\sim_{\mathrm{SE}}B
 \quad\Longleftrightarrow\quad
 L_k(G_A)\sim_{\mathrm{grMor}}L_k(G_B).
\end{equation*}
In the broader formulation surveyed in
\cite[Conjecture~8.8.2]{CortinasHazratSurvey}, shift equivalence is also expected to be equivalent to gauge-equivariant strong Morita equivalence of the graph $C^*$-algebras and to equivalence of certain singularity categories (the latter connection goes back to Chen and Yang \cite{ChenYang}). The conjecture is attractive because it would give a computable, field-independent classification of a large class of graded algebras, would place the algebraic and analytic pictures on the same footing, and it would identify a familiar invariant from symbolic dynamics as a complete invariant for graded Morita equivalence.

There seems to be significant evidence in its favour when strongly connected components of the graph are simple cycles. Hazrat proved the
pointed conjecture for polycephaly graphs \cite{HazratClassification}. Ara and Pardo established a general weak form for finite sinkless and sourceless graphs, in which one is allowed to twist by a locally inner automorphism of the degree-zero algebra \cite{AraPardo}. On the Morita-theoretic side, the conjecture was verified in several restricted settings, including meteor graphs \cite{CordeiroGillaspyGoncalvesHazrat}, certain connected three-vertex graphs \cite{HazratPacheco}, and connected finite sinkless and sourceless graphs whose Leavitt path algebras have Gelfand--Kirillov dimension three \cite{DoHazratNam}. Ara, Do, and Nam obtained a further positive result for graphs with disjoint cycles containing exactly three cycles whose lengths are pairwise coprime \cite{AraDoNam}. The pointed graded-isomorphism conjecture for graphs with disjoint cycles was also addressed in a recent preprint of Va\v{s} \cite{VasDisjointCycles}. Arnone also proved that the pointed ordered graded Grothendieck module classifies Leavitt path algebras of finite primitive graphs up to graded homotopy equivalence \cite{ArnoneHomotopy}.

A second line of work addresses the lifting problem. Arnone \cite{ArnoneLifting} and Va\v{s} \cite{VasFull} proved, by different methods, that the graded Grothendieck functor is full in the relevant sense. That is, positive pointed module maps lift to graded homomorphisms. This is an important part of Hazrat's original functorial formulation, but it does not show that mutually inverse maps of graded $K$-groups can be lifted to mutually inverse algebra maps. Abrams, Ruiz, and Tomforde introduced bridging bimodules and proved that a module shift equivalence satisfying additional compatibility conditions induces a graded Morita equivalence \cite{ARTRecasting}. The second-named author together with Aguyar Brix, Hazrat, and Ruiz obtained a pointed refinement, where unitally aligned module shift equivalence induces a graded isomorphism \cite{BDHR}. Separately, the combinatorial compatible, aligned, and balanced shift-equivalence relations, defined using path bijections, were shown to coincide with strong shift equivalence for finite essential matrices  \cite{BilichDorOnRuizCorrespondences}. However, despite all of these developments, a central question still remained: does arbitrary shift equivalence already force graded Morita equivalence?

The main result of this paper shows, identifies the graded Morita equivalence of Leavitt path algebras with strong shift equivalence rather than with shift equivalence.

\begin{theoremintro}\label{thm:main}
Let $A$ and $B$ be finite essential adjacency matrices, and let $k$ be a field. The Leavitt path algebras $L_k(G_A)$ and $L_k(G_B)$ are graded Morita equivalent if and only if $A$ and $B$ are strong shift equivalent.
\end{theoremintro}

Thus, we see that graded Morita equivalence remembers the full conjugacy relation of the associated shifts of finite type, even though the diagonal subalgebra is not included in its definition. Equation \eqref{eq:Kgr-SE-intro} and Theorem~\ref{thm:main} show that the obstruction to classification by $K_0^{\mathrm{gr}}$ is the gap between shift equivalence and strong shift equivalence.

To witnes this gap, we appeal to the primitive Kim--Roush examples \cite{KimRoush}, which then yield a counterexample to the Graded Morita equivalence conjecture. Their adjacency matrices are shift equivalent, so the corresponding ordered graded Grothendieck groups are isomorphic over every field, but they are not strong shift equivalent, so that the associated Leavitt path algebras are not graded Morita equivalent. 

\begin{corollaryintro}
There are primitive matrices $A$ and $B$ such that, over every field $k$, the ordered $\mathbb Z[x,x^{-1}]$-modules $\Kgr(\Lk{G_A})$ and $\Kgr(\Lk{G_B})$ are isomorphic, but $\Lk{G_A}$ and $\Lk{G_B}$ are not graded Morita equivalent.
\end{corollaryintro}

By the work of Bratteli and Kishimoto \cite{BratteliKishimoto}, we know that for primitive adjacecy matrices the corresponding Cuntz--Krieger C*-algebras are nevertheless stably gauge-equivariantly isomorphic. Hence the algebraic and analytic versions of the proposed classification diverge already for primitive matrices, and we see that the process of completion of the Leavitt path algebras into graph C*-algebras loses information about the underlying adjacency matrices.

The failure of the graded Morita equivalence conjecture is not repaired by recording the order unit. Using successive in-splittings with nonempty incoming parts, we realize the image of the unit in the dimension-group isomorphism by another primitive matrix. This produces a pair of primitive matrices for which there is a unit-preserving ordered $\mathbb Z[x,x^{-1}]$-module isomorphism of graded $K_0$ groups, while the two algebras are not graded isomorphic. Thus, the pointed graded-isomorphism conjecture fails over every field as well.

\begin{corollaryintro}
There are finite primitive adjacency matrices $A,C$ such that, for every field $k$, there is a unit-preserving ordered $\mathbb Z[x,x^{-1}]$-module isomorphism $\Kgr(\Lk{G_A})\cong\Kgr(\Lk{G_C})$, but $\Lk{G_A}$ and $\Lk{G_C}$ are not graded isomorphic as graded rings.
\end{corollaryintro}

Since strong shift equivalence coincides with graded Morita equivalence in this class, for finite essential matrices we see that the graded Morita equivalence problem for Leavitt path algebras contains exactly the unresolved decidability problem for shifts of finite type. In particular, graded Morita equivalence does not reduced to the decidable dimension-group invariant.

Our proof is inspired by previous works where strong shift equivalence was recovered from appropriately defined equivalence relations on graph C*-algebras and C*-correspondences associated to adjacency matrices \cite{CarlsenDorOnEilers, BilichDorOnRuizCorrespondences}. 

Let $C\subseteq D$ be finite-dimensional unital $k$-algebras. Every finite-dimensional right module has a composition series, and the Jordan--H\"older theorem makes the multiplicity $[M:S]$ of a simple module $S$ among the composition factors of $M$ independent of the chosen series. Thus, after ordering the simple right modules
$S_1,\ldots,S_r$ of $C$ and $T_1,\ldots,T_s$ of $D$, the inclusion recovers a well-defined nonnegative integer matrix
\[
 \mathsf M(C,D)_{ij}
   =[\operatorname{Res}_C^D T_j:S_i].
\]
The restricted module need not be semisimple, which makes the Jordan--H\"older multiplicity an appropriate notion for our purposes. Since restriction of scalars is exact, these matrices multiply along chains in the sense that if $C\subseteq D\subseteq H$ is a finite dimensional algebra inclusion, then
\[
 \mathsf M(C,H)=\mathsf M(C,D)\mathsf M(D,H).
\]
For the consecutive algebras $Q_n(G_A)\subseteq Q_{n+1}(G_A)$ whose direct limit is $\Lk{G_A}_0$, the simple blocks are indexed by terminal vertices, and the resulting restriction matrix is precisely $A$ in the sense that its $(v,w)$ entry counts the possible last edges from $v$ to $w$. The proof constructs interlacing chains of finite-dimensional algebras for
which this multiplicativity gives factorizations $A_r=R_rS_r$ and
$A_{r+1}=S_rR_r$. In this way, non-negative integer matrices implementing strong shift equivalence arise from inclusions.

Suppose now that $A$ and $B$ are essential and $\varphi:L_k(G_A)\longrightarrow L_k(G_B)$ is a graded $k$-algebra isomorphism.  Write
$p_u=\varphi(u)$, $x_e=\varphi(e)$, and $y_e=\varphi(e^*)$.  If $K$ is a finite-dimensional subalgebra of $L_k(G_B)_0$ containing the $p_u$, put $\widehat K=\sum_{e,f\in G_A^1}x_eKy_f$. The main observation is that this enlargement has an internal reconstruction in terms of coefficients in the sense that
\[
 y_ebx_f\in K,
 \ \ \text{and} \ \  
 b=\sum_{e,f\in G_A^1}x_e(y_ebx_f)y_f.
\]
Intersecting one fixed $Q_n(G_B)$ algebra of $G_B$ with the successive $Q_{n+1}(G_A)$ algebras coming from $G_A$ then produces a finite chain
\[
 K_r\subseteq K_{r+1}\subseteq\widehat K_r
       \subseteq\widehat K_{r+1}.
\]
Restriction multiplicities along this chain then yield factorizations
$A_r=R_rS_r$ and $A_{r+1}=S_rR_r$ whose endpoint matrices are precisely $A$ and $B$ Thus, the argument gives strong shift equivalence, so that a graded $k$-algebra
isomorphism $L_k(G_A)\cong_{\mathrm{gr}}L_k(G_B)$ implies
$A\sim_{\mathrm{SSE}}B$.

The reduction from graded Morita equivalence to the aforementioned result is achieved by applying some technical results. The homogeneous graded Morita theorem of Abrams--Ruiz--Tomforde \cite{ARTMorita} gives a graded ring isomorphism $L_k(G_A)\cong_{\mathrm{gr}}eM_d(L_k(G_B))e$ where $e$ is full in $M_d(L_k(G_B)_0)$. We then place $e$ in one finite-dimensional $Q_n(G_B)$ algebra and diagonalize it there. The same diagonalization is used first to normalize the action on the coefficient field and then, by conjugating the resulting path idempotents with the paths themselves, to replace the corner by the Leavitt path algebra of a certain graph with "row-repeated" adjacency matrix $B^{\mathbf m}$. Its adjacency matrix is elementary strong shift related to $B$. The associated graph has no sinks, although it may have sources, but a short extension of the result which allows the second adjacency matrix to include sources then establishes a proof of the main result.

The paper is organized as follows. Section~\ref{sec:prelim} recalls the passage from matrices to graphs, the basic graded structure of Leavitt path algebras, composition series and Jordan--H\"older multiplicities, and the extraction of adjacency matrices from the $Q_n(G_B)$ filtration. In Section~\ref{sec:isomorphism} we prove that the existence of a graded $k$-algebra isomorphism $L_k(G_A)\cong_{\mathrm{gr}}L_k(G_B)$ implies $A\sim_{\mathrm{SSE}}B$ by the finite-chain argument just described. Section~\ref{sec:technical} treats source deletion, full degree-zero corners, scalar normalization, and the reduction from graded Morita equivalence to graded isomorphism.  Finally, Section~\ref{sec:conjectures} applies the main theorem to the original Kim--Roush matrices and provides the pointed counterexample by in-splittings which preserve primitivity of the adjacency matrices.

\vspace{6pt}

\textbf{AI disclosure statement.} GPT-6 was used to explore proof strategies for this paper. More precisely, it was used to upgrade a solution of \cite[Question 1]{ARTRecasting} obtained by the authors to a strategy for constructing a strong shift equivalence from a graded isomorphism of Leavitt path algebras. The key observation was to construct a specific sequence of inclusions of finite dimensional subalgebras of the zero-graded components of a Leavitt path algebra, some of which are not semisimple, and to use the Jordan--H\"older multiplicity matrices of inclusions to recover matrices implementing a strong shift equivalence sequence for the underlying essential adjacency matrices $A$ and $B$. The authors subsequently checked and wrote the arguments in their final form and take full responsibility for all claims and proofs in the paper.

\section{Preliminaries} \label{sec:prelim}

We collect here the notation and the finite-dimensional algebra preliminaries used in the proof. Standard references for Leavitt path
algebras are \cite{AbramsArandaPino,AAMS}. For composition series and the Jordan--H\"older theorem we refer, for instance, to \cite{LamFirstCourse}.

Let $V$ and $W$ be finite sets, and let $R=[R_{vw}]_{v\in V,w\in W}$ be a matrix over $\mathbb Z_{\geq 0}$. We write
\[
 E_R=\{(v,i,w):v\in V,\ w\in W,\ 0\leq i < R_{vw}\},
\]
with source and range maps $s(v,i,w)=v$ and $r(v,i,w)=w$. When $V=W$, this gives a directed graph $G_R=(V,E_R,s,r)$. Thus there are exactly $R_{vw}$ edges from $v$ to $w$. Throughout the
paper, rows are indexed by sources and columns by ranges.

A square nonnegative integer matrix will be referred to as an \emph{adjacency matrix}. An adjacency matrix is called \emph{essential} if it has no zero rows and no zero columns. Equivalently, an adjacency matrix $A$ is essential if and only if the graph $G_A$ has no sinks and no sources. For adjacency matrices $A$ and $B$, we write $A\ESSE B$ if there are rectangular nonnegative integer matrices $R,S$ such that $A=RS$ and $B=SR$. Strong shift equivalence, denoted $A\SSE B$, is the equivalence relation
generated by elementary strong shift equivalence. We write $A\SE B$ if there are an integer lag $\ell\geq 1$ and rectangular nonnegative integer matrices $R,S$ satisfying
\[
 AR=RB,\ \ \ \ BS=SA,\ \ \ \ RS=A^\ell,\ \ \ \ SR=B^\ell.
\]

Let $G=(G^0,G^1,s,r)$ be a finite directed graph and let $k$ be a field. The Leavitt path algebra $L_k(G)$ is the universal $k$-algebra generated by pairwise orthogonal idempotents $v$, $v\in G^0$, together with edges $e$ and ghost edges $e^*$, $e\in G^1$, subject to
\begin{align*}
 s(e)e=e=er(e),& \ \ \ \ r(e)e^*=e^*=e^*s(e),\\
 e^*f=\delta_{e,f}r(e), & \ \ \ \ 
 v=\sum_{s(e)=v}ee^*,\  \text{with} \ s^{-1}(v) \neq \emptyset.
\end{align*}
Since $G$ is finite, $L_k(G)$ is unital with
$1=\sum_{v\in G^0}v$. For a path $\alpha=e_1\cdots e_n$, we put $|\alpha|=n$ and $\alpha^*=e_n^*\cdots e_1^*$. Vertices are regarded as paths of length zero. The path normal form says that $L_k(G)$ is spanned by the elements $\alpha\beta^*$ with $r(\alpha)=r(\beta)$.

The standard $\mathbb Z$-grading is determined by requiring that
\[
 \deg(v)=0,\qquad \deg(e)=1,\qquad \deg(e^*)=-1.
\]
For a graded ring $R=\bigoplus_{n\in\mathbb Z}R_n$, let $\operatorname{Gr}\text{-}R$ denote the category of graded right
$R$-modules and degree-zero homomorphisms. Two graded rings $R$ and $S$ are \emph{graded Morita equivalent} if there is a categorical equivalence $F: \operatorname{Gr}\text{-}R \longrightarrow \operatorname{Gr}\text{-}S$ such that $F(M(n)) = F(M)(n)$. Recall that $R$ is \emph{strongly graded} if $R$ is \emph{strongly graded} if $R_nR_m=R_{n+m}$ for every $m,n\in\mathbb Z$. If $G$ is finite and has no sinks, then $L_k(G)$ is strongly graded (see \cite[Chapter~3]{AAMS}).

Let $C$ be a finite-dimensional $k$-algebra and let $M$ be a
finite-dimensional right $C$-module. A \emph{composition series} for $M$ is a finite chain of submodules 
\begin{equation*}
 0=M_0\subsetneq M_1\subsetneq\cdots\subsetneq M_\ell=M
\end{equation*}
for which each quotient $M_i/M_{i-1}$ is simple. The quotients are the \emph{composition factors} of the series. Every finite-dimensional module has a composition series: one may choose a maximal proper submodule and argue by induction on $\dim_k M$.

The Jordan--H\"older theorem says that any two composition series of $M$ have the same length and the same simple factors, counted with multiplicity, up to a permutation. Consequently, if $S$ is a simple right $C$-module, the integer $[M:S]$ defined as the number of composition factors isomorphic to $S$ is independent of the chosen series. We call it the \emph{Jordan--H\"older
multiplicity} of $S$ in $M$. It is additive on short exact sequences. That is, if
\[
 0\longrightarrow M'\longrightarrow M\longrightarrow M''
 \longrightarrow 0
\]
is exact, then we have additivity
\begin{equation*}
 [M:S]=[M':S]+[M'':S].
\end{equation*}
Indeed, a composition series for $M'$ together with the inverse images of a composition series for $M''$ gives a composition series for $M$.

Now let $C\subseteq D$ be finite-dimensional unital $k$-algebras. If $T$ is a finite-dimensional right $D$-module, then $\operatorname{Res}_C^D T$ is finite-dimensional over $k$, and hence has finite length as a $C$-module. The Jordan--H\"older multiplicities make sense even without any semisimplicity assumption.

Choose ordered complete sets $S_1,\ldots,S_r$ and $T_1,\ldots,T_s$ of representatives for the simple right modules over $C$ and $D$,
respectively. The multiplicity matrix of the
inclusion $C\subseteq D$ is
\begin{equation*}
 \Mmult{C}{D}_{ij}
   =\big[\operatorname{Res}_C^D T_j:S_i\big]
 \in\mathbb Z_{\geq 0}.
\end{equation*}
Changing the order of the chosen simple modules merely permutes the rows or columns of this matrix. When $C$ is split semisimple, restriction is semisimple and the multiplicity matrix has the more familiar meaning
\[
 \operatorname{Res}_C^D T_j
   \cong\bigoplus_{i=1}^r S_i^{\,\Mmult{C}{D}_{ij}}.
\]
In particular, if $C\cong\bigoplus_{i=1}^rM_{c_i}(k)$ and $D\cong\bigoplus_{j=1}^sM_{d_j}(k)$, then $\Mmult{C}{D}$ is the usual inclusion matrix, or Bratteli matrix,
of $C\subseteq D$ (see \cite{BratteliAF}). The $j$-th column records how the simple module of the $j$-th block of $D$ decomposes on restriction to $C$, and one has $d_j=\sum_{i=1}^r c_i\Mmult{C}{D}_{ij}$. We will use the following basic multiplicativity property.

\begin{lem} \label{lem:multiplicity-product}
Let $C\subseteq D\subseteq H$ be finite-dimensional unital
$k$-algebras. After fixing ordered representatives for the simple right modules of the three algebras, we have that
\begin{equation*}
 \Mmult{C}{H}=\Mmult{C}{D}\Mmult{D}{H}.
\end{equation*}
\end{lem}

\begin{proof} Let $S_1,\ldots,S_r$ and $T_1,\ldots,T_s$ and $U_1,...,U_{\ell}$ be representatives for the simple right modules over $C$ and $D$ and $H$ respectively. As a $D$-module, the class of $U_k$ is
\[
 [\operatorname{Res}_D^H U_k]
   =\sum_j[\operatorname{Res}_D^H U_k:T_j][T_j]
\]
in the Grothendieck group of finite-length $D$-modules. Restriction of scalars is exact, so after restricting to $C$ and using additivity, we have that
\[
 [\operatorname{Res}_C^H U_k : S_i] = \sum_j
 [\operatorname{Res}_D^H U_k:T_j] [\operatorname{Res}_C^D T_j:S_i].
\]
Thus, we have obtained the desired equation in the lemma.
\end{proof}

The following elementary consequence will produce strong shift equivalences later on.

\begin{lem} \label{lem:interlacing-inclusions}
Suppose $C_0\subseteq C_1\subseteq D_0\subseteq D_1$ are finite-dimensional unital $k$-algebras. Assume that the simple
modules of $D_i$ have been labelled by those of $C_i$, for $i=0,1$, and that under these labels we have $\Mmult{D_0}{D_1}=\Mmult{C_0}{C_1}$. Denote
\[
 A_i=\Mmult{C_i}{D_i},\qquad
 R=\Mmult{C_0}{C_1},\qquad
 S=\Mmult{C_1}{D_0}.
\]
Then $A_0=RS$ and $A_1=SR$, so that $A_0\ESSE A_1$.
\end{lem}

\begin{proof}
Apply Lemma~\ref{lem:multiplicity-product} first to
$C_0\subseteq C_1\subseteq D_0$ and then to
$C_1\subseteq D_0\subseteq D_1$. The first chain gives $A_0=RS$; the second gives $A_1=SR$ by the multiplicativity property and the assumption that $\Mmult{D_0}{D_1}=\Mmult{C_0}{C_1}$.
\end{proof}

Let $A$ be a finite essential adjacency matrix, and put $G=G_A$. For $n\geq 0$ define
\begin{align*}
 Q_n(G)&=\spn\{\alpha\beta^*:|\alpha|=|\beta|=n,
                         \ r(\alpha)=r(\beta)\}\subseteq \Lk{G}_0,\\
 H_n^+(G)&=\spn\{\alpha\beta^*:|\alpha|=n+1,\ |\beta|=n,
                         \ r(\alpha)=r(\beta)\}\subseteq \Lk{G}_1,\\
 H_n^-(G)&=\spn\{\alpha\beta^*:|\alpha|=n,\ |\beta|=n+1,
                         \ r(\alpha)=r(\beta)\}\subseteq \Lk{G}_{-1}.
\end{align*}
We omit $G$ from the notation when no confusion can arise. For
$v\in G^0$, let $G^nv$ be the set of paths of length $n$ ending at $v$, and put
\[
 z_v^{(n)}=\sum_{\alpha\in G^nv}\alpha\alpha^*.
\]
Since $G$ has no sources, $G^nv$ is nonempty for every $v$ and $n \in \mathbb{Z}_{\geq 0}$.

\begin{lem}\label{lem:path-pair}
Let $A$ be a finite essential adjacency matrix.
\begin{enumerate}
\item The algebras $Q_n(G)$ are unital finite-dimensional, split semisimple, $Q_n(G)\subseteq Q_{n+1}(G)$ and
\[
 \Lk{G}_0=\bigcup_{n\geq 0}Q_n(G), \ \ \ \Lk{G}_1=\bigcup_{n\geq 0}H_n^+(G), \ \ \ \Lk{G}_{-1}=\bigcup_{n\geq 0}H_n^-(G).
\]
\item For every $n\geq 0$ we have,
\begin{align*}
 H_n^-Q_nH_n^+&\subseteq Q_n,&
 H_n^+Q_nH_n^-&\subseteq Q_{n+1},&
 H_n^-Q_{n+1}H_n^+&\subseteq Q_n.
\end{align*}
\item There is a canonical block decomposition
\begin{equation}\label{eq:path-pair-blocks}
 Q_n(G)\cong\bigoplus_{v\in G^0}M_{|G^nv|}(k),
\end{equation}
in which $z_v^{(n)}$ is the identity of the $v$-block. Moreover, if
$h\in H_n^+(G)$, or $h'\in H_n^-(G)$,
\begin{equation}\label{eq:path-block-intertwining}
 h z_v^{(n)}=z_v^{(n+1)}h, \ \ \ h'z_v^{(n+1)} = z_v^{(n)}h'.
\end{equation}
\end{enumerate}
\end{lem}

\begin{proof}
At fixed length, the elements $\alpha\beta^*$ with
$r(\alpha)=r(\beta)$ form matrix units, which gives the canonical block decomposition. Since $G$ has no sinks, repeated use of the Cuntz--Krieger relation yields
\begin{equation}\label{eq:path-extension}
 \alpha\beta^*
   =\sum_{s(f)=r(\alpha)}\alpha f(\beta f)^*.
\end{equation}
This proves $Q_n\subseteq Q_{n+1}$ and, in the same way,
$H_n^\pm\subseteq H_{n+1}^\pm$. The path normal form gives the three exhaustions. The products in (ii) and equation \eqref{eq:path-block-intertwining} follow by multiplying path matrix units, since every nonzero product is again a path matrix unit of the indicated length.
\end{proof}

We now identify the inclusion matrix of two consecutive $Q_n$
algebras. This is where the original adjacency matrix appears inside the degree-zero algebra.

\begin{prp} \label{prop:path-multiplicity}
Let $A$ be a finite essential matrix. Label the simple right modules of $Q_n(G_A)$ and $Q_{n+1}(G_A)$ by their terminal vertices their canonical block decompositions. Then 
$$
\Mmult{Q_n(G_A)}{Q_{n+1}(G_A)}=A.
$$
\end{prp}

\begin{proof}
Fix $w\in G_A^0$, and let $T_w^{(n+1)}$ be the simple right module of the $w$-block of $Q_{n+1}(G_A)$. It has a standard basis indexed by the paths of length $n+1$ ending at $w$. Every such path has a unique factorization as $\alpha f$ for $|\alpha|=n$ and $f$ an edge from $v$ to $w$. For a fixed such edge $f$, let $X_f$ be the span of the basis vectors indexed by the paths $\alpha f$. The inclusion $Q_n(G_A)\subseteq Q_{n+1}(G_A)$ is given on matrix units by equation \eqref{eq:path-extension}, so it follows that $X_f$ is a $Q_n(G_A)$-submodule and that $X_f\cong T_v^{(n)}$, the simple module of the $v$-block of $Q_n(G_A)$. Distinct last edges
give a direct sum decomposition
\[
 \operatorname{Res}_{Q_n}^{Q_{n+1}}T_w^{(n+1)}
 \cong
 \bigoplus_{v\in G_A^0}
 \big(T_v^{(n)}\big)^{A_{vw}}.
\]
Hence the multiplicity of the $v$-simple in the restriction of the
$w$-simple is exactly 
$$
\#\{f\in G_A^1:s(f)=v,\ r(f)=w\}=A_{vw}.
$$ 
\end{proof}

\section{Graded isomorphisms for essential matrices} \label{sec:isomorphism}

Let $A$ and $B$ be finite matrices where $A$ is essential and $B$ with no zero rows. Suppose that $\varphi:\Lk{G_A}\longrightarrow \Lk{G_B}$ is a graded $k$-algebra isomorphism. Write
\begin{equation*}
 p_u=\varphi(u),\ \ \ x_e=\varphi(e),\ \ \ y_e=\varphi(e^*).
\end{equation*}
Inside $\Lk{G_B}$ these elements continue to satisfy the Cuntz-Krieger relations
\begin{align*}
 p_up_v=\delta_{u,v}p_u, \ \ \ y_ex_f=\delta_{e,f}p_{r(e)},\ \ \ p_u=\sum_{s(e)=u}x_ey_e.
\end{align*}
In particular, we have that 
\begin{equation}\label{eq:sum-xy}
 1=\sum_{e\in G_A^1}x_ey_e.
\end{equation}
Denote $P=\spn\{p_u:u\in G_A^0\}$, and for a path $\alpha=e_1\cdots e_r$, write $x_\alpha=x_{e_1}\cdots x_{e_r}$ and
$y_\alpha=y_{e_r}\cdots y_{e_1}$ for the corresponding path and ghost path.

For our given matrix $B$ with no zero rows, define $Q_n(G_B)$ and $H_n^\pm(G_B)$ by the same formulas as in the previous section. The filtration, exhaustion, and product statements of Lemma~\ref{lem:path-pair} remain valid, and the only difference is that a terminal-vertex block in equation \eqref{eq:path-pair-blocks} may generally be smaller or absent.

\begin{lem}\label{lem:edge-enlargement}
Let $K\subseteq \Lk{G_B}_0$ be a finite-dimensional unital subalgebra
containing $P$, and define $\widehat K=\sum_{e,f\in G_A^1}x_eKy_f$. Then we have the following
\begin{enumerate}
\item $\widehat K$ is a finite-dimensional unital subalgebra of $\Lk{G_B}_0$, and $b\in\widehat K$ if and only if $y_ebx_f\in K\quad\text{for every }e,f\in G_A^1$. In that case we have the reconstruction
\begin{equation*}
 b=\sum_{e,f\in G_A^1}x_e(y_ebx_f)y_f.
\end{equation*}
\item Choose, for every $u\in G_A^0$, an edge $\epsilon_u$ with
$r(\epsilon_u)=u$, and denote
\begin{equation*}
 s=\sum_{u\in G_A^0}x_{\epsilon_u},\ \ \ 
 t=\sum_{u\in G_A^0}y_{\epsilon_u},\ \ \  q=st.
\end{equation*}
Then $ts=1$, the idempotent $q$ is full in $\widehat K$, and the map $\Theta_K:K\longrightarrow q\widehat Kq$ given by $a\longmapsto sat$ is an algebra isomorphism.
\item If $K\subseteq L \subseteq \Lk{G_B}_0$ are two finite dimensional unital subalgebras containing $P$, use the full-corner isomorphisms $\Theta_K$ and $\Theta_L$ to label the simple modules of $\widehat K$ and $\widehat L$ by those of $K$ and $L$,
respectively. Then we have that $\Mmult{\widehat K}{\widehat L}=\Mmult{K}{L}$.
\end{enumerate}
\end{lem}

\begin{proof}
It is clear that $\widehat K$ is an algebra, and it is unital because $1=\sum_{e\in G_A^1}x_ep_{r(e)}y_e$. If $b\in\widehat K$, then Cuntz-Krieger relations show that every coefficient
$y_ebx_f$ lies in $K$. Conversely, if all these coefficients lie in $K$, insert equation \eqref{eq:sum-xy} on both sides of $b$ to obtain the desired reconstruction equation in item (1).

The chosen edges have distinct range vertices, and hence we have $ts=\sum_{u\in G_A^0}p_u=1$, so that $q=st$ is an idempotent. The map $\Theta_K$ is a homomorphism, and $b\mapsto tbs$ is its inverse. Indeed, if $b\in q\widehat Kq$, then $tbs$ is a sum of coefficients of $b$ and therefore belongs to $K$ by item (1), while $ts=1$ gives the two inverse identities.

To see that $q$ is full, fix $e\in G_A^1$ and put $u=r(e)$. Both
$x_ey_{\epsilon_u}$ and $x_{\epsilon_u}y_e$ belong to $\widehat K$, and $x_ey_e=(x_ey_{\epsilon_u})q(x_{\epsilon_u}y_e)$. Summing over $e$ and using equation \eqref{eq:sum-xy} gives $1\in\widehat Kq\widehat K$.

For item (3), the same idempotent $q$ is full in $\widehat K$ and
$\widehat L$, and the two corner isomorphisms are compatible with the inclusion $K\subseteq L$. The functor $T\mapsto Tq$ is exact and implements the full-corner equivalence. Applying it to a composition series for the restriction of a simple $\widehat L$-module preserves all Jordan--H\"older multiplicities. Under the isomorphisms $\Theta_K$ and $\Theta_L$, the corner inclusion $q\widehat Kq\subseteq q\widehat Lq$ is exactly $K\subseteq L$, which proves the state equality of matrices.
\end{proof}

For $r\geq 0$, let
\begin{equation*}
 P_r=\varphi(Q_r(G_A))
     =\spn\{x_\alpha y_\beta:|\alpha|=|\beta|=r,
                              \ r(\alpha)=r(\beta)\}.
\end{equation*}
Since $A$ is essential, we have that
\begin{equation*}
 P=P_0\subseteq P_1\subseteq\cdots,
 \ \ \ \text{and} \ \ \ \Lk{G_B}_0=\bigcup_{r\geq 0}P_r.
\end{equation*}
Separating the first edge from a path, it follows that $\widehat{P_r}=P_{r+1}$. Now, choose $n$ large enough that
\begin{equation}\label{eq:choose-n}
 p_u\in Q_n(G_B),\ \ \
 x_e\in H_n^+(G_B),\ \ \ 
 y_e\in H_n^-(G_B)
\end{equation}
for every $u\in G_A^0$ and $e\in G_A^1$. Denote $C=Q_n(G_B)$, and note that since $C$ is finite-dimensional and $\Lk{G_B}_0=\bigcup_rP_r$, we may choose $N\geq 1$
with $C\subseteq P_N$. Define $K_r=C\cap P_r$ for $0\leq r\leq N$, so that $K_0=P$ and $K_N=C$.

\begin{lem}\label{lem:finite-chain}
For $0\leq r<N$, we have that $K_{r+1}=C\cap\widehat K_r$. Thus,
\begin{equation}\label{eq:four-chain}
 K_r\subseteq K_{r+1}\subseteq\widehat K_r
       \subseteq\widehat K_{r+1}.
\end{equation}
Moreover, we have $\widehat C=Q_{n+1}(G_B)$.
\end{lem}

\begin{proof}
Let $c\in K_{r+1}=C\cap P_{r+1}$. Since $\widehat{P_r}=P_{r+1}$, we have that $y_ecx_f\in P_r$ for every $e,f\in G_A^1$. On the other hand, equation \eqref{eq:choose-n} and Lemma~\ref{lem:path-pair}(ii) give
\[
 y_ecx_f\in H_n^-Q_n(G_B)H_n^+\subseteq Q_n(G_B)=C.
\]
Thus all the coefficients $y_ecx_f$ belong to $K_r$, and
Lemma~\ref{lem:edge-enlargement}(i) gives $c\in\widehat K_r$. This proves that $K_{r+1} \subseteq C\cap\widehat K_r$. The reverse inclusion follows from $\widehat K_r\subseteq\widehat P_r=P_{r+1}$. The rest of the inclusions in the lemma are immediate.

To show $\widehat C=Q_{n+1}(G_B)$, we note first that Lemma~\ref{lem:path-pair}(ii) yields $\widehat C\subseteq Q_{n+1}(G_B)$. Conversely, if
$c\in Q_{n+1}(G_B)$, we have
\[
 y_ecx_f\in H_n^-Q_{n+1}(G_B)H_n^+\subseteq Q_n(G_B)=C
\]
for every $e,f\in G_A^1$, so that the coefficient test
in Lemma \ref{lem:edge-enlargement}(i) yields $c\in\widehat C$.
\end{proof}

\begin{thm}
\label{thm:isomorphism-rigidity}
Let $A$ and $B$ be finite essential matrices. If $\Lk{G_A}\cong_{\mathrm{gr},k}\Lk{G_B}$ then $A\SSE B$.
\end{thm}

\begin{proof}
For $0\leq r\leq N$, define $A_r=\Mmult{K_r}{\widehat K_r}$, where the simple modules of $\widehat K_r$ are relabelled by those of
$K_r$ using Lemma~\ref{lem:edge-enlargement}(ii). For $r<N$, put
\[
 R_r=\Mmult{K_r}{K_{r+1}},\qquad
 S_r=\Mmult{K_{r+1}}{\widehat K_r}.
\]
Then, Lemma~\ref{lem:edge-enlargement}(iii) yields
\[
 \Mmult{\widehat K_r}{\widehat K_{r+1}}
   =\Mmult{K_r}{K_{r+1}}=R_r.
\]
Applying Lemma~\ref{lem:interlacing-inclusions} to the four algebras in equation \eqref{eq:four-chain}, we obtain that $A_r = R_rS_r$ and $A_{r+1}=S_rR_r$. Thus every consecutive pair $A_r,A_{r+1}$ is elementary strong shift equivalent.

It remains to identify the endpoints. Since $K_0=P$ and
$\widehat K_0=P_1$, label the simple modules of $P$ by the vertices of
$G_A$. The algebra $P_1$ is a direct sum of matrix blocks indexed by
$w\in G_A^0$, and a standard basis for the simple module of the $w$-block
is indexed by the edges ending at $w$. By Cuntz-Krieger relations, right multiplication by $p_u$ selects precisely the basis vectors indexed by
edges from $u$ to $w$. Hence
\[
 \Mmult{P}{P_1}_{u,w}
   =\#\{e:s(e) =u,\ r(e)=w\}=A_{uw}.
\]
The full-corner label agrees with this terminal-vertex label because
$\Theta_P(p_w)=x_{\epsilon_w}y_{\epsilon_w}$ is a minimal idempotent in
the $w$-block. Therefore, we have $A_0=A$.

At the other endpoint, $K_N=C=Q_n(G_B)$ and
$\widehat K_N=Q_{n+1}(G_B)$. Since the element $s$ belongs to $H_n^+(G_B)$,
equation \eqref{eq:path-block-intertwining} gives $s z_v^{(n)}=z_v^{(n+1)}s$. Consequently,
\[
 \Theta_C(z_v^{(n)})
   =s z_v^{(n)}t
   =z_v^{(n+1)}q.
\]
Since $q$ is full in $Q_{n+1}(G_B)$, we have that $z_v^{(n+1)} q \neq 0$ for every $v$. It is the identity of the $v$-summand of the full corner $qQ_{n+1}(G_B)q$ Thus the label induced by the full-corner equivalence is precisely the natural terminal-vertex label $v$. Hence, the full-corner labels agree with the natural terminal-vertex labels on $Q_n(G_B)$ and $Q_{n+1}(G_B)$. Proposition~\ref{prop:path-multiplicity} therefore yields $A_N=\Mmult{Q_n(G_B)}{Q_{n+1}(G_B)}=B$. Thus, we have obtained $A\SSE B$.
\end{proof}

\section{Graded Morita equivalence in terms of strong shift equivalence} \label{sec:technical}

We now reduce from a graded Morita equivalence to the situation treated in the previous section. The only auxiliary graph which may have sources is a graph arising from a full corner. Define
\[
 I_n(C)=\{v\in G_C^0:\text{there is a path of length $n$ ending at $v$}\}
\]
and let $C_n$ be the corner of $C$ with respect to the vertices of $I_n(C)$, so that $G_{C_n}$ is the subgraph of $G_C$ induced by $I_n(C)$.

\begin{lem}\label{lem:essential-core}
The sets $I_n(C)$ decrease and eventually stabilize. For every $n$ we have that $C_n\ESSE C_{n+1}$. If $I_N(C)=I_{N+1}(C)=I$, then The corner of $C$ by the set $I$ is essential, and $C\SSE C_N$. Moreover, under the terminal-vertex labels, we have
\begin{equation*}
 \Mmult{Q_N(G_C)}{Q_{N+1}(G_C)}=C_N.
\end{equation*}
\end{lem}

\begin{proof}
Deleting the first edge of a path shows that $I_{n+1}(C)\subseteq I_n(C)$, so the sets eventually stabilize by finiteness of the graph $C$. Put $J=I_{n+1}(C)$. A vertex in $I_n(C)\setminus J$ receives no edge from a vertex in $I_n(C)$, and hence the corresponding column of $C_n$ is zero. Let $R$ be the $I_n(C) \times J$ corner of $C_n$ and take $S$ to be the coordinate-selection matrix given by $S_{j,u} = \delta_{j,u}$ for $j\in J$ and $u\in I_n(C)$. Then
\[
 RS=C_n,\ \ \ SR=C_{n+1},
\]
which is an elementary strong shift equivalence.

At a stable stage, say for $N\in \mathbb{N}$ such that $I_N(C) = I_{N+1}(C)$, every vertex of $G_{C_N}$ receives an edge from
$G_{C_N}$. It also emits an edge inside $G_{C_N}$. Indeed, if $u\in I$ and $e$ is an edge from $u$ to $w$ in $G_C$, then a path of length $N$ ending at $u$, followed by this edge, shows that $w\in I_{N+1}(C)=I$. Thus $C_N$ is essential, and the successive elementary equivalences give $C\SSE C_N$.

Finally, the blocks of $Q_N(G_C)$ and $Q_{N+1}(G_C)$ are both indexed by $I$. The proof of Proposition~\ref{prop:path-multiplicity} applies verbatim and shows that the restriction multiplicity from the $w$-block is indexed by edges from $v$ to $w$ with $v,w\in I$. This is exactly the desired conclusion.
\end{proof}

\begin{prp}
\label{prop:source-tolerant}
Let $A$ be a finite essential adjacency matrix and let $C$ be a finite adjacency matrix with no zero rows. If $\Lk{G_A}\cong_{\mathrm{gr},k}\Lk{G_C}$ then $A\SSE C$.
\end{prp}

\begin{proof}
Let $\varphi:\Lk{G_A}\to\Lk{G_C}$ be a graded $k$-algebra isomorphism. Choose $n$ large enough so that all image vertices lie in $Q_n(G_C)$, all image edges and ghosts lie in $H_n^+(G_C)$ and $H_n^-(G_C)$, respectively, and $I_n(C)=I_{n+1}(C)$. Run the construction of the previous section with the fixed algebra $Q_n(G_C)$. The proof of Lemma \ref{lem:edge-enlargement} uses only essentiality of $A$, while the proof of Lemma \ref{lem:finite-chain} uses only the $Q_n(G_C)$ product
relations, which remain valid for $G_C$.

The initial endpoint is still $A$. At the other endpoint,
Lemma~\ref{lem:essential-core} gives
\[
 A_N=\Mmult{Q_n(G_C)}{Q_{n+1}(G_C)}=C_n.
\]
Thus $A\SSE C_n$. Since by the same lemma we have $C\SSE C_n$, we get that $A\SSE C$.
\end{proof}

Suppose $B$ is a finite essential matrix. Let $\mathbf m=(m_v)_{v\in G_B^0}$ be positive integers. Define a matrix
$B^{\mathbf m}$ whose rows and columns are indexed by the set $\{v_i:v\in G_B^0,\ 1\leq i\leq m_v\}$ by specifying $(B^{\mathbf m})_{v_i,w_j}=B_{vw}\delta_{j,1}$. Equivalently, for every edge $f$ from $v$ to $w$ of $G_B$ and every $1\leq i\leq m_v$, the graph $G_{B^{\mathbf m}}$ has an edge $f_i$ from $v_i$ to $w_1$. Thus the row of $v$ is repeated $m_v$ times, while every repeated edge enters the first copy of its range. The graph $G_{B^{\mathbf m}}$ has no sinks, but it may have sources. Let $m=\sum_vm_v$ and define
\begin{equation}\label{eq:row-corner-idempotent}
 D_{\mathbf m}=\bigoplus_{v\in G_B^0} \operatorname{diag}
 \big(\underbrace{v,\ldots,v}_{m_v\text{ times}}\big)
 \in M_m(\Lk{G_B}).
\end{equation}

\begin{lem}\label{lem:row-repetition}
There is a graded $k$-algebra isomorphism $\Lk{G_{B^{\mathbf m}}}\cong_{\mathrm{gr},k} D_{\mathbf m}M_m(\Lk{G_B})D_{\mathbf m}$, and $B^{\mathbf m}\ESSE B$.
\end{lem}

\begin{proof}
Index the rows and columns of $M_m(\Lk{G_B})$ by the vertices $v_i$. Define for vertices $v_i$ and edges $f_i$ with $s(f_i) = v_i$ and $r(f_i) = w_1$,
\begin{align*}
 v_i&\longmapsto E_{v_i,v_i}v,\\
 f_i&\longmapsto E_{v_i,w_1}f,\\
 f_i^*&\longmapsto E_{w_1,v_i}f^*.
\end{align*}
These elements satisfy the Cuntz--Krieger relations for
$G_{B^{\mathbf m}}$, so they define a graded homomorphism $\Psi:\Lk{G_{B^{\mathbf m}}} \longrightarrow D_{\mathbf m}M_m(\Lk{G_B})D_{\mathbf m}$. This homomorphism is injective by the graded uniqueness theorem, since every vertex has nonzero image.

For surjectivity, observe that for $1\leq i,j\leq m_v$,
\begin{equation}\label{eq:row-matrix-units}
 \sum_{s(f)=v}\Psi(f_i)\Psi(f_j^*)=E_{v_i,v_j}v.
\end{equation}
The first copies $v_1$ and the edges $f_1$ generate a copy of
$\Lk{G_B}$ in the copy of the first corner. Hence the image contains
$E_{v_1,w_1}b$ for every $b\in v \Lk{G_B} w$. Multiplying by the matrix units in equation \eqref{eq:row-matrix-units} gives $E_{v_i,w_j}b$ for all $i,j$ and all $b\in v\Lk{G_B} w$. These elements span the corner.

Next, we show $B^{\mathbf m}\ESSE B$. Define rectangular matrices $R,S$ by setting $R_{v,w_j}=\delta_{v,w}\delta_{j,1}$ and $S_{v_i,w}=B_{vw}$. Then
\[
 RS=B,
 \ \ \text{and} \ \ 
 (SR)_{v_i,w_j}=B_{vw}\delta_{j,1}
                 =(B^{\mathbf m})_{v_i,w_j}.
\]
Thus, $B^{\mathbf m}\ESSE B$.
\end{proof}

For a graded ring $D$, put $Z_0(D)=Z(D)\cap D_0$, the center of $D$ in $D_0$.

\begin{lem}\label{lem:center}
Let $B$ be a finite essential adjacency matrix. For each connected component $W$ of the underlying undirected graph of $G_B$, put $c_W=\sum_{v\in W}v$. Then
\begin{equation*}
 Z_0(\Lk{G_B})=\bigoplus_Wkc_W.
\end{equation*}
If $e\in M_d(\Lk{G_B}_0)$ is full in $M_d(\Lk{G_B}_0)$, then
\begin{equation*}
 Z_0(eM_d(\Lk{G_B})e)
   =\bigoplus_Wk e(c_WI_d).
\end{equation*}
\end{lem}

\begin{proof}
Each $c_W$ is central. Conversely, let $z\in Z_0(\Lk{G_B})$, and choose $n$ with $z\in Q_n(G_B)$. Since $z$ is central in both $Q_n(G_B)$ and
$Q_{n+1}(G_B)$, there are scalars $\lambda_v,\mu_v\in k$ such that
\begin{equation}\label{eq:center-two-stages}
 z=\sum_v\lambda_vz_v^{(n)}
  =\sum_v\mu_vz_v^{(n+1)}.
\end{equation}
For each $w$, choose a nonzero element of $H_n^+(G_B)$ whose length-$n$
and length-$(n+1)$ sides both end at $w$. Centrality and equation
\eqref{eq:path-block-intertwining} yield $\lambda_w=\mu_w$. If
$f$ is an edge from $v$ to $w$, choose a path $\gamma$ of length $n$ ending at $v$ and multiply equation \eqref{eq:center-two-stages} by the nonzero idempotent $\gamma f(\gamma f)^*$. This yields $\lambda_v=\mu_w=\lambda_w$. Thus the scalars are constant on each connected component. Since
$\sum_{v\in W}z_v^{(n)}=c_W$, the description of $Z_0(\Lk{G_B})$ follows.

Next, denote $R=M_d(\Lk{G_B})$. The standard center isomorphism $Z(R)\longrightarrow Z(eRe)$ for a full corner is given by $z\longmapsto ez$, and preserves degree zero. Indeed, if
$\sum_i a_i e b_i=1$ with $a_i,b_i\in R_0$, the inverse sends
$t\in Z(eRe)$ to $\sum_i a_itb_i$, which has degree zero whenever $t$
does. Since $Z(R)=Z(\Lk{G_B})I_d$, the descriptio of $Z_0(eM_d(\Lk{G_B})e)$ follows from the already established description of $Z_0(\Lk{G_B})$. Fullness ensures that every $e(c_WI_d)$ is nonzero.
\end{proof}

\begin{prp}
\label{prop:corner-replacement}
Let $A$ and $B$ be finite essential adjacency matrices, and let
$e\in M_d(\Lk{G_B}_0)$ be full in $M_d(\Lk{G_B}_0)$. If there is a
graded ring isomorphism $\varphi:\Lk{G_A}\longrightarrow eM_d(\Lk{G_B})e$, then there are positive integers $m_v$ for each $v\in G_B^0$ which yield a graded $k$-algebra isomorphism $\Lk{G_A}\cong_{\mathrm{gr},k}\Lk{G_{B^{\mathbf m}}}$.
\end{prp}

\begin{proof}
Choose a finite fullness relation $\sum_j a_jeb_j=1$ for $a_j,b_j\in M_d(\Lk{G_B}_0)$. Since $\Lk{G_B}_0=\bigcup_tQ_t(G_B)$, there is $t$ such that $e$ and all the
$a_j,b_j$ belong to $M_d(Q_t(G_B))$. Thus $e$ is full in the
finite-dimensional semisimple algebra $M_d(Q_t(G_B))$. Since $B$ is essential, we have
\[
 Q_t(G_B)\cong\bigoplus_{v\in G_B^0}M_{|G_B^tv|}(k).
\]
Let $m_v$ be the rank of the $v$-block of $e$ in $M_d(Q_t(G_B))$.
Fullness gives $m_v>0$ for every $v \in G_B^0$. With $m=\sum_vm_v$, diagonalizing the idempotent separately in each matrix block, we find an invertible $g\in M_d(Q_t(G_B))$ such that
\begin{equation*}
 q=geg^{-1}
   =\sum_{j=1}^{m}E_{a_j,a_j}\alpha_j\alpha_j^*,
\end{equation*}
where $|\alpha_j|=t$ for all $j$, the pairs $(a_j,\alpha_j)$ are distinct, and $v$ occurs exactly $m_v$ times among the vertices $r(\alpha_j)$. Conjugation by $g$ is degree preserving, so after replacing $e$ by $q$ and $\varphi$ by $\operatorname{Ad}(g)\circ\varphi$, we may assume that $e=q$. 

The diagonal idempotent $q$ is fixed by every field automorphism induced on $M_d(\Lk{G_B})$. For a connected component $U$ of the underlying undirected graph of $G_A$ write $c_U^A = \sum_{u\in U}u$, and define $c_W^B$ similarly for $G_B$. By Lemma~\ref{lem:center} the primitive central idempotents in the degree-zero centers of the source and target are, respectively, $c_U^A$ and $q(c_W^B I_d)$. Hence, $\varphi$ induces a bijection between connected components $U$ of $G_A$ and connected components $W$ of $G_B$. For each matched pair, the restriction of $\varphi$ to $k c_U^A$ is onto $k q(c_W^B I_d)$ and is a unital ring isomorphism. Consequently, there is a field automorphism $\sigma_W \in \Aut(k)$ such that $\varphi(\lambda c_U^A) = \sigma_W(\lambda)q(c_W^B I_d)$ for every $\lambda \in k$. Summing over the components gives
\begin{equation}\label{eq:scalar-action}
 \varphi(\lambda1)
   =\sum_W\sigma_W(\lambda)q(c^B_WI_d).
\end{equation}
Let $\tau_{\sigma_W}$ denote the coefficientwise ring automorphism of
$M_d(\Lk{G_B})$ induced by $\sigma_W$ which fixes all graph generators, for each component $W$. Since $q$ and $c_W^B I_d$ are sums of matrices of path idempotents with coefficients $1$, they must therefore fixed by the automorphism $\tau_{\sigma_W}$. Therefore, for $b\in qM_d(\Lk{G_B})q$ the map
\begin{equation}\label{eq:scalar-correction}
 \Theta(b)=\sum_W q(c^B_WI_d)\tau_{\sigma_W^{-1}}(b),
\end{equation}
is a graded ring automorphism of the corner. Equations
\eqref{eq:scalar-action} and \eqref{eq:scalar-correction} therefore yield $\Theta\varphi(\lambda1)=\lambda q$. Replacing $\varphi$ by $\Theta\circ\varphi$, we may therefore assume that $\varphi$ is $k$-linear.

Now define $D=\operatorname{diag}(r(\alpha_1),\ldots,r(\alpha_m))\in M_m(\Lk{G_B})$, as well as rectangular matrices $T\in M_{d\times m}(\Lk{G_B})$ and
$U\in M_{m\times d}(\Lk{G_B})$ by $T_{a_j,j}=\alpha_j$ and $U_{j,a_j}=\alpha_j^*$, with all other entries zero. Since the paths have the same length and the pairs $(a_j,\alpha_j)$ are distinct, we get that $UT=D$ and $TU=q$. Every nonzero entry of $T$ has degree $t$, and every nonzero entry of $U$
has degree $-t$. Hence we get a graded $k$-algebra isomorphism $DM_m(\Lk{G_B})D \longrightarrow qM_d(\Lk{G_B})q$ given by $z\longmapsto TzU$, with inverse given by $w\mapsto UwT$. The diagonal idempotent $D$ contains the vertex $v$ exactly $m_v$ times, so $D=D_{\mathbf m}$ in the notation of equation \eqref{eq:row-corner-idempotent}. Lemma~\ref{lem:row-repetition} now identifies $DM_m(\Lk{G_B})D$ with $\Lk{G_{B^{\mathbf m}}}$. Composing the inverse our isomorphism and the isomorphism with $\varphi$ gives the required graded $k$-algebra isomorphism.
\end{proof}

We now arrive at the main result of our paper.

\begin{thm}\label{thm:main-classification}
Let $A$ and $B$ be finite essential adjacency matrices, and let $k$ be a field. $L_k(G_A)$ and $L_k(G_B)$ are graded Morita equivalent if and only if $A$ and $B$ are strong shift equivalent.
\end{thm}

\begin{proof}
That $A\SSE B$ implies that $L_k(G_A)$ and $L_k(G_B)$ are graded Morita equivalent is a known construction, which can be found for instance in \cite[Proposition~15(2)]{HazratDynamics}.

Leavitt path algebras associated to finite essential matrices are strongly graded. Abrams--Ruiz--Tomforde show that, for strongly graded rings, graded equivalence is the same as homogeneous graded equivalence \cite[Proposition~3.5 and Corollary~3.6]{ARTMorita}. Their homogeneous graded Morita theorem \cite[Theorem~4.7]{ARTMorita} gives an integer $d\geq 1$, and an idempotent $e\in M_d(\Lk{G_B}_0)$ which is full in this degree-zero algebra, so that we have a graded ring isomorphism $\Lk{G_A}\cong_{\mathrm{gr}}eM_d(\Lk{G_B})e$. Proposition~\ref{prop:corner-replacement} and the second part of Lemma \ref{lem:row-repetition} give positive integers
$\mathbf m=(m_v)$ and a graded $k$-algebra isomorphism $\Lk{G_A}\cong_{\mathrm{gr},k}\Lk{G_{B^{\mathbf m}}}$ so that $B^{\mathbf m}\ESSE B$. The graph $G_{B^{\mathbf m}}$ has no sinks. Therefore Proposition~\ref{prop:source-tolerant} gives
$A\SSE B^{\mathbf m}$, and hence $A\SSE B$.
\end{proof}

\section{Hazrat's graded classification conjectures are false} \label{sec:conjectures}

For a finite essential adjacency matrix $A$, and write $u_A=[\Lk{G_A}]\in\Kgr(\Lk{G_A})$. We use graded right modules and the suspension convention $M(1)_j=M_{j+1}$, with $x[M]=[M(1)]$. The ordered dimension module of
$A$ is
\[
 \mathcal D_A=\varinjlim(\mathbb Z^{G_A^0},A^\transpose),
 \ \ \
 \mathcal D_A^+=\{[z,n]:z\in\Zp^{G_A^0},\ n\geq 0\}.
\]
The natural identification is 
$$
(\Kgr(\Lk{G_A}),\Kgr(\Lk{G_A})^+,u_A)
 \cong (\mathcal D_A,\mathcal D_A^+,[\one,0])
$$
where $x[z,n]=[A^\transpose z,n]$. This follows from the finite $Q_n(G_A)$ stages and Dade's equivalence (see for instance \cite{HazratDynamics,AraPardo}). The action is consistent with
\[
 v\Lk{G_A}(1)\cong\bigoplus_{s(e)=v}r(e)\Lk{G_A}.
\]

If $R,S$ implement a shift equivalence between $A$ and $B$ of lag $l \geq 1$, then the induced map $\widehat R:\mathcal D_A\longrightarrow\mathcal D_B$ given by $\widehat R([z,n])=[R^\transpose z,n]$ is an ordered $\mathbb Z[x,x^{-1}]$-module isomorphism, with inverse $x^{-\ell}\widehat S$. In particular, shift equivalence gives an isomorphism of ordered graded Grothendieck groups over every field. The two graded classification assertions predict:
\begin{enumerate}
\item isomorphic ordered graded Grothendieck groups imply graded Morita equivalence of the Leavitt path algebras;
\item an isomorphism of these ordered modules taking $u_A$ to $u_B$ implies graded isomorphism of the Leavitt path algebras.
\end{enumerate}

\begin{cor}\label{cor:unpointed}
There are primitive adjacency matrices $A$ and $B$ such that, over every field $k$, the ordered modules $\Kgr(\Lk{G_A})$ and
$\Kgr(\Lk{G_B})$ are isomorphic, but $\Lk{G_A}$ and $\Lk{G_B}$ are not graded Morita equivalent.
\end{cor}

\begin{proof}
Kim and Roush constructed primitive nonnegative integer matrices $A,B$ which are shift equivalent but not strong shift equivalent
\cite[Section~7]{KimRoush}. The matrices are essential, and shift
equivalence gives the ordered module isomorphism above. If
$\Lk{G_A}$ and $\Lk{G_B}$ were graded Morita equivalent,
Theorem~\ref{thm:main-classification} would give $A\SSE B$, a contradiction.
\end{proof}

We adjust the unit using only in-splittings with nonempty incoming
parts, so that every matrix in the construction remains essential throughout.

\begin{lem}\label{lem:unit-realization}
Let $B$ be a finite primitive adjacency matrix indexed by $V_B$, and suppose that $B$ is not a permutation matrix. Let $m\in\mathbb{Z}_{>0}^{V_B}$ where $B^t m\ge m$. Then there is a primitive adjacency matrix $C$ such that $G_C$ is obtained from $G_B$ by in-splittings with nonempty incoming parts, so that $C\sim_{\mathrm{SSE}}B$. Moreover, there is an ordered $\mathbb{Z}[x,x^{-1}]$-module isomorphism $\eta\colon\mathcal{D}_C\longrightarrow\mathcal{D}_B$ satisfying $\eta([\mathbf{1}_C,0])=[m,0]$.
\end{lem}

\begin{proof}
We first describe one in-splitting. Let $D$ be a primitive adjacency matrix which is not a permutation matrix. Split one vertex into two copies by partitioning its incoming edges into two nonempty parts, so that every other vertex has one copy. Assign each old edge to the copy of its range specified by the partition, and duplicate it once for each copy of its source.

Let $T$ have rows indexed by the new vertices and columns indexed by the old vertices, with $T_{i,a}=1$ when $i$ is a copy of $a$, and zero otherwise. Let $U_{a,j}$ count the old edges from $a$ assigned to the new range vertex $j$. The old and new adjacency matrices satisfy $D=UT$ and $D'=TU$. Thus the in-splitting gives an elementary strong shift equivalence. 

If $D$ is primitive, then so is $D'$. Indeed, for every $r\ge0$, we have that $(D')^{r+1}=TD^rU$. Given new vertices $i,j$, let $a$ be the parent of $i$, and choose $b$ with $U_{b,j}>0$. Such a $b$ exists because every column of $U$ is nonzero. Since $D$ is irreducible, there is an $r\ge0$ with $(D^r)_{a,b}>0$, and hence
\[
    (D')^{r+1}_{i,j}=(TD^rU)_{i,j}>0.
\]
Thus $D'$ is irreducible. If $D^r>0$, then $TD^rU>0$, since every row of $T$ and every column of $U$ is nonzero. Hence $D'$ is also primitive.

The associated ordered $\mathbb{Z}[x,x^{-1}]$-module isomorphism $\widehat{T}\colon\mathcal{D}_{D'}\longrightarrow\mathcal{D}_D$ is given by $\widehat{T}([z,n])=[T^t z,n]$. Its inverse is $x^{-1}\widehat{U}$, or explicitly $\widehat{T}^{-1}([z,n])=[U^t z,n+1]$. The factorizations above ensure that these maps are well defined, mutually inverse, commute with $x$, and preserve the positive cones.

We now construct the required sequence of in-splittings. Every
current vertex retains a label in $V_B$. Let $Q_D$ be the matrix
whose $(i,v)$-entry is $1$ precisely when the current vertex $i$ is
labelled $v$, and define $c=Q_D^t\mathbf{1}_D$. Thus $c_v$ is the number of current copies of $v$. Initially $D=B$, $Q_D=I$, and $c=\mathbf{1}_B$.

Every current vertex labelled $u$ has exactly $B_{u,v}$ outgoing
edges whose ranges are labelled $v$. This property is preserved by
each in-splitting, so $DQ_D=Q_DB$. In particular, $(B^tc)_v$ is the total number of edges entering the $c_v$ copies of $v$. Primitivity gives
\[
    B^tc=Q_D^tD^t\mathbf{1}_D
        \ge Q_D^t\mathbf{1}_D=c.
\]
We increase $c$ to $m$, one coordinate at a time, while maintaining $c\le m$. Suppose $0<c\le m$, $B^tc\ge c$, and $c\ne m$, then some $v\in V_B$ satisfies $c_v<m_v$ and $(B^tc)_v>c_v$. To provev this claim, suppose this fails, and set $d=m-c$ and $J=\{v\in V_B:d_v>0\}$. Then $J\ne\varnothing$ and $(B^tc)_v=c_v$ for every $v\in J$. Consequently, for $v\in J$ we have
\[
    (B^td)_v=(B^tm)_v-(B^tc)_v
             \ge m_v-c_v=d_v>0
\]
Take $a=\min_{v\in J}c_v$, and denote $S=\{v\in J:c_v=a\}$. Fix $v\in S$. Since $(B^td)_v>0$, there is a predecessor $u\in J$
with $B_{u,v}>0$. By the definition of $a$, we get that
\[
    a=c_v=\sum_w B_{w,v}c_w
      \ge B_{u,v}c_u
      \ge a.
\]
All inequalities are therefore equalities. This implies that
$B_{u,v}=1$, $c_u=a$, and $B_{w,v}=0$ for all $w\ne u$. Thus every vertex of $S$ has exactly one incoming edge, whose source
also belongs to $S$. In particular, $S$ is nonempty and receives no edges from its complement. Irreducibility forces $S=V_B$. Every vertex of $G_B$ therefore has exactly one incoming edge. Strong connectivity then forces $G_B$ to be a directed cycle, contradicting the assumption that $B$ is not a permutation matrix. This proves the claim.

Now, for a $v$ supplied by the claim, one of its $c_v$ copies receives at least two edges. Split that copy using two nonempty incoming parts. This increases $c_v$ by one and preserves $c\leq m$. After exactly $\sum_v(m_v-1)$ steps, the construction terminates with $c=m$. Let $C$ be the primitive adjacency matrix of the resulting $G_C$ with vertices $V_C$.

The product of the successive matrices $T$ is the matrix $Q$ with rows indexed by $V_C$ and columns indexed by $V_B$, where $Q_{i,v}=1$ precisely when the vertex $i$ is labelled $v$. Thus $Q^\transpose\one=m$. Composing the isomorphisms $\widehat T$ gives $\eta$ and yields $\eta([\one,0])=[Q^\transpose\one,0]=[m,0]$.
\end{proof}

\begin{cor} \label{cor:pointed}
There are finite primitive adjacency matrices $A,C$ such that, for every field $k$, there is a unit-preserving ordered $\mathbb Z[x,x^{-1}]$-module isomorphism $\Kgr(\Lk{G_A})\cong\Kgr(\Lk{G_C})$, but $\Lk{G_A}$ and $\Lk{G_C}$ are not graded isomorphic as graded rings.
\end{cor}

\begin{proof}
Let $A$ and $B$ be the primitive Kim--Roush matrices from \cite[Section~7]{KimRoush}, which are not cyclic permutation matrices. Since they are shift equivalent, there are some nonnegative matrices $R,S$ implementing a shift equivalence of lag $\ell\geq 1$ between $A$ and $B$. Take $m=R^\transpose\one$. Every $m_v$ is positive. Indeed, a zero column of $R$ would give a zero column of
$SR=B^\ell$, which is impossible for a primitive matrix. Since $A$ is primitive, $A^\transpose\one\geq\one$. Transposing $AR=RB$ therefore gives
\[
 B^\transpose m
   =R^\transpose A^\transpose\one
   \geq R^\transpose\one=m.
\]
Apply Lemma~\ref{lem:unit-realization} to obtain a primitive matrix $C$ and an isomorphism $\eta:\mathcal D_C\to\mathcal D_B$ with $\eta([\one,0])=[m,0]$. The ordered module isomorphism $\beta=\eta^{-1}\widehat R:\mathcal D_A\longrightarrow\mathcal D_C$ preserves the distinguished units because $\widehat R([\one,0])=[m,0]$. Under the identification of dimension triples with graded $K_0$ triples, this is a pointed graded $K$-theory isomorphism over every field.

On the other hand we have that $C\SSE B$. Thus, if $\Lk{G_A}$ and $\Lk{G_C}$ were graded Morita equivalent, Theorem~\ref{thm:main-classification} would give $A\SSE C\SSE B$, contrary to the Kim--Roush construction. A graded ring isomorphism is therefore impossible as well.
\end{proof}

\end{document}